\documentclass[twoside, 12pt]{article}
\usepackage[utf8]{inputenc}
\usepackage[english]{babel}
\usepackage{amsthm}
\usepackage{geometry}
\usepackage{amstext}
\usepackage{amssymb}
\usepackage{amsmath}
\usepackage{setspace}
\usepackage{cmap}
\usepackage[pdftex]{graphicx}
\usepackage{hyperref}
\newtheorem{thm}{Theorem}
\newtheorem{lem}{Lemma}
\theoremstyle{definition}
\newtheorem{remark}{Remark}
\begin{document}
\title{Survival Probability of Markov Linear Reccurence Sequence in a Random Environment: Subcritical Case} 
\author{Alexander Shklyaev \footnote{Steklov Mathematical Institute of Russian Academy of Sciences, Moscow, Russia, alexander.shklyaev@math.msu.ru}} 

\maketitle
\begin{abstract}
The Markov linear recurrence sequence in a random environment (MRSRE) is a particular case of Markov chain on non-negative integers. It is a natural generalization of several models with branching. We consider MRSRE with absorption at zero and study the survival time. We introduce a classification of subcritical MRSRE corresponding to that of branching process in a random environment. We show that general R-positivity theory can be applied to this model in a strongly subcritical case and prove the analogues of the Kolmogorov theorem and the Yaglom theorem. In other words, we describe the asymptotic behavior of the survival probability over a long time and the conditional distribution of the sequence, conditioned on the survival event. The results are applied to particular branching models.
\end{abstract}

\section{Introduction} 

Let $\boldsymbol\eta=(\eta_1,\eta_2,\dotsc)$ be a sequence of i.i.d. random variables (r.v.) called the environment. Assume that $\{B_n\}$ is a sequence of random variables such that 
$\mathcal{F}_n:=\sigma(B_1,\dotsc, B_n, \eta_1, \dotsc, \eta_n)$ is independent of $\mathcal{G}_{\ge n}:=\sigma(\eta_{n+1},\eta_{n+2},\dotsc)$. 

We say that the sequence $Y_{n+1} = A_{n+1} Y_n + B_{n+1}$, $n\ge 0$, where $A_{n} = g_{\xi}(\eta_{n})$ for some measurable function $g_{\xi}$, is a Markov Linear Recurrence Sequence in a Random Environment (MRSRE) if 
$\{Y_{n}\}$ is a Markov chain in a random environment $\boldsymbol\eta$ (so, for a fixed $\boldsymbol\eta$ it is a non-homogeneous Markov chain with a transition matrix $Q_{\eta}$ for some family $\{Q_y, y\in \mathbb{R}\}$). 

We always assume the following asymptotic negligibility condition on $\{B_n\}$: for some $h$ there exists $\widetilde{h}=\widetilde{h}(h)\in (0,h)$ such that
\begin{equation}
\label{BnCond}
{\bf E}_{\boldsymbol\eta}(|B_{n+1}|^{h}|Y_n)\le \kappa_{n+1} e^{h\xi_{n+1}} Y_n^{\widetilde{h}},\ Y_n>0,
\end{equation}
for $\kappa_n=g_{\kappa}(\eta_{n};h)$ for some measurable $g_{\kappa}(\cdot;h)$. Note that by the Lyapunov inequality if~(\ref{BnCond}) holds for some $h$, then it holds for any $h_0\in (0,h]$ with $\widetilde{h}(h_0)=h_0 \widetilde{h}(h)/h$. Thus, we always assume that~(\ref{BnCond}) holds for all $h$ in some interval $(0,\widehat{h}]$. For the same reason, we always assume that $h-\widetilde{h}(h)$ is increasing in $h$ on $(0,\widehat{h}]$ and that $\widetilde{h}(h)$ is continuous.

This model generalizes many known branching models, including the Galton-Watson branching process (BPGW), the branching processes with immigration (BPI) or migration (BPM), the branching processes in a random environment (BPRE), the branching processes in a random environment with immigration (BPREI) or migration (BPREM), the bisexual branching processes (BBP), and the bisexual branching process in a random environment (BBPRE). The models close to MRSRE were introduced in~\cite{ShkUp}. In \cite{ShkBBPRE} a BBPRE was represented in such a form. All the models above and their connection with MRSRE will be given in the last section.

The random walk $\{S_n\}$ with the steps $\xi_i=\ln A_i$, $i\ge 1$, is called the associated random walk for the MRSRE $\{Y_n\}$. If $\mu={\bf E}\xi<0$, then we call MRSRE subcritical, if $\mu=0$, we call it critical, otherwise it is called supercritical.
Assume that the Markov chain $\{Y_n\}$ has the state space $D\cup\{0\}$, where $D\subseteq \mathbb{N}$, $\{0\}$ is an absorbing state, the chain reduced to $D$ is irreducible and aperiodic, and the probability of absorption is positive. In this case we call $\{Y_n\}$ a positive zero-absorbing MRSRE.

We consider a subcritical MRSRE and study the probabilities ${\bf P}(Y_n>0)$ and ${\bf P}(Y_n=i|Y_n>0)$, $i\in \mathbb{N}$, as $n\to\infty$. This problem has a long history for different branching models. For BPGW the corresponding results are known as the Kolmogorov theorem (see \cite{KolmogorovBP}) and the Yaglom theorem (see \cite{Yaglom}). 
For the BPRE the subcritcal case is divided into three regimes, our work is related to the strongly subcritical one (see Chapter 9 of~\cite{VatKerst} for results for the BPRE).
For the BPI stopped at zero the similar results were obtained in~\cite{SubcritImSeneta}; for the BPM stopped at zero, in \cite{SubcritM}; for BPREI stopped at zero, in \cite{SubcritIm}.
However, for the BBP or BBPRE no such results are known. We also do not know such results for the BPREM stopped at zero. 
We prove that for MRSRE there is a general result of this kind, using the general theorem from work \cite{Quasi}.

It is known (see~\cite{Vere-Jones}) that 
there exists
\begin{equation}
\label{rho}
\rho:=\lim_{n\to\infty}\sqrt[n]{{\bf P}(Y_n=j|Y_n=i)}\in (0,1),
\end{equation}
and $\rho$ is independent of $i,j\in D$. The fact that $\rho<1$ is a corollary of Lemma 2 of \cite{Seneta-Vere-Jones}. 

Let $R(h)={\bf E}e^{h\xi}$ be finite for $h\in [0,h^+)$. Assume that $\rho^+:=\inf_{[0,h^+)} R(h)<\rho$. The corresponding MRSRE is called strongly subcritical.
 Then there exists $h^*_+$ such that $R(h^*_+)=\rho$ and $R'(h^*_+)<0$.

The main result of this work is the following. Let $T$ be the time until absorption of $\{Y_n\}$ and let~(\ref{BnCond}) hold for some $h> h^*_+$. Then we show that ${\bf P}(T>n)\sim C \rho^n$ for some positive $C$ and 
$${\bf P}(Y_n=k|T>n)\to p_k^*$$
 for some distribution $\{p_k^*\}$ on $D$. Section~\ref{SubcriticalYn} is devoted to this case.

The second result is the following. Assume that $\{Y_n\}$ is an irreducible subcritical MRSRE (so, $\{0\}$ is not absorbing). In Section~\ref{NoAbsorption} we show that $\{Y_n\}$ is geometrically ergodic under some moment conditions. 
 
In Section~\ref{Applications} we consider particular cases of MRSRE: BPRE, BPREI, BBPRE and formulate the theorems for these cases. Some of them are new for these models. We allow $\xi$ to be degenerate, therefore, the same result is true for the BPRE, BBP, BPM. Our general result is weaker than known results for particular models, but for several models the results are new.
\section{Preliminaries}
\subsection{Cramer condition}
Let us give a short introduction to the Cramer measure transform (for a brief discussion see \cite{Borov}). Let $\xi$, $\xi_i$, $i\ge 1$, be i.i.d. r.v. on the probability space $(\Omega,\mathcal{F},{\bf P})$. Assume that $\xi$ satisfies the right-hand Cramer condition:
$$
R(h):={\bf E}e^{h\xi}<+\infty,\ h\in [0,\widehat{h}).
$$
The function $R(h)$ is convex infinitely smooth function. Let $m(h):= (\ln R)'(h)$, $\sigma^2(h):=m'(h)$, $h\in (0,\widehat{h})$. Let $\xi^{(h)}$ be a r.v. with the distribution
$$
{\bf P}(\xi^{(h)}\in A) = R(h)^{-1} {\bf E}(e^{h \xi}; A),\ h\in [0,\widehat{h}).
$$
The corresponding distribution is called conjugate. Note that ${\bf E}\xi^{(h)}=m(h)$, $\operatorname{Var}\xi^{(h)}=\sigma^2(h)$, $h\in (0,\widehat{h})$.

 If $\xi^{(h)}$, $\xi^{(h)}$, $i>0$, are i.i.d. r.v., then for any natural $n$ we can define the measure
$$
{\bf P}_n^{(h)}(A):= {\bf P}\left(\left(\xi^{(h)}_1,\dotsc, \xi^{(h)}_n\right)\in B\right),\quad 
A = \{\omega:\ (\xi_1(\omega),\dotsc, \xi_n(\omega))\in B\},\ B\in \mathcal{B}(\mathbb{R}^n).
$$
By the Kolmogorov extension theorem there exists the corresponding measure ${\bf P}^{(h)}$ on $\sigma(\xi_i, i>0)$. Finally, we extend this measure to $\mathcal{F}$ by the relation
$$
{\bf P}^{(h)}(A):= {\bf E}^{(h)}{\bf P}(A|\xi_1,\dotsc, \xi_n),\quad A\in \mathcal{F},
$$
where ${\bf E}^{(h)}$ is the expectation on $\sigma(\xi_1,\xi_2,\dotsc, \xi_n)$ with respect to ${\bf P}^{(h)}$. 
\subsection{Martingale ineqiality}
We need a number of inequalities close to the classical Marcinkiewicz–Zygmund inequality. 
\begin{thm}
\label{BahrTh}
Let $(X_n,\mathcal{F}_n)$ be a martingale difference sequence, ${\bf E}X_n^{h}<+\infty$, $h>1$. Then 
$$
{\bf E}|\sum_{i=1}^{n} X_i|^h\le C \left(\sum_{i=1}^{n} \left({\bf E}|X_i|^h\right)^{\min(2/h,1)}\right)^{\min(2/h, 1)},
$$
where $C$ is some constant. For $h\in (1,2)$ $C\le 2$.
\end{thm}
For the case $h\in (1,2]$ the theorem is proved in~\cite{Pinelis}. For $h>2$ it's Theorem 2 of \cite{NagaevMartingales}. The first versions in the case of i.i.d. variables were obtained by von Bahr and Esseen (\cite{Bahr}) and Dharmadhikari and Jogdeo (\cite{Dhar}).
\subsection{$R$-positivity}
We need some general facts about Markov chains from work~\cite{Quasi}. Let $\{X_n\}$ be a Markov chain on $\mathbb{N}_{0}=\mathbb{N}\cup \{0\}$. Let $P=(p_{i,j}, i,j\ge 0)$ be the transition matrix of $\{X_n\}$ and $P^+=(p_{i,j}, i>0, j>0)$. Assume that $\{0\}$ is the only absorbing state, accessible from some positive state. Let $\{X_n\}$ be irreducible and aperiodic on $\mathbb{N}$ (in other words, the matrix $P^+$ is aperiodic and irreducible). If the absorption time $T$ is finite a.s., then we say that $\{X_n\}$ is a positive Markov chain with certain extinction.
Note that instead of $\mathbb{N}\cup \{0\}$ we can consider the state space $A\cup\{0\}$, where $A$ is a subset of $\mathbb{N}$.

For a positive Markov chain with certain extinction it is known (see \cite{Vere-Jones}), that there exists $R\in [1,+\infty)$ such that independently of $i,j\in \mathbb{N}$
$$
R^{-1} = \lim_{n\to\infty} {\bf P}(X_n=j|X_0=i)^{1/n}.
$$ 
\begin{thm}[\cite{Quasi}, Theorem 1]
\label{MainQSTh}
Let $\{X_n\}$ be a positive Markov chain with certain extinction.
Suppose that there exists a set $K\subset \mathbb{N}$ such that:
\begin{enumerate}
\item[QS1)] for some positive $C_1$, $\varepsilon$ and all $n\in \mathbb{N}$ 
$${\bf P}(X_i\not\in \{0\}\cup K, 1\le i\le n|X_0=j)\le C_1 (R+\varepsilon)^{-n},\ j\in K;$$
\item[QS2)] for some $j\in K$, $C_2>0$ and for all $i\in K$ and $n\in \mathbb{N}$
$${\bf P}(T>n|X_0=i)\le C_2 {\bf P}(T>n|X_0=j);$$
\item[QS3)] for some finite subset $K_1\subset \mathbb{N}$, a natural number $n_0$ and a positive constant $C_3$ for all $i\in K$ 
$${\bf P}(\exists j\le n_0:\ X_j\in K_1|X_0=i)\ge C_3.$$
\end{enumerate}
Then
\begin{enumerate}
\item[1)] $\pi P_{+} =  R^{-1} \pi$, where $\pi$ is a probability distribution on $\mathbb{N}$;
\item[2)] for some non-negative function $f:\mathbb{N}\to \mathbb{R}$ we have $P_{+} f = R^{-1} f$, and $\sum_{i>0} f(i) \pi_i=1$;
\item[3)] for all $i\in \mathbb{N}$ we have
\begin{equation}
\label{SurviveProbability}
\lim_{n\to\infty} R^{-n} {\bf P}(T>n|X_0=i)=f(i);
\end{equation}
\item[4)] for all $i,j\in \mathbb{N}$ we have
\begin{equation}
\label{LocalSurvProb}
\lim_{n\to\infty} {\bf P}(X_n=j|X_0=i, T>n)= \pi_j.
\end{equation}
\end{enumerate}
\end{thm}
In work~\cite{QSDExp} (see Corollary 3) it is shown that 
\begin{equation}
\label{RAsConvRadius}
R=\inf\{u: {\bf E}u^{T}=+\infty\}.
\end{equation}
This gives another expression for $R$. In a recent preprint~\cite{QSDNew} the quasi-stationary distribution $\pi$ is characterized as well.
\begin{thm}{(Theorem 2.2 and Corollary 2.5 of~\cite{QSDNew}).}
\label{QSDForm}
Under the conditions of~Theorem~\ref{MainQSTh} 
$$
\pi_i = \frac{R-1}{{\bf E} \left(\left.R^{T_{i}}T_i; T_i<T_0\right|X_0=i\right)}, i\in \mathbb{N},
$$
where $T_i$ is the first return to the state $i$.
\end{thm}

Note that for some $c$ there exists a nonnegative non-zero sequence $\{x\}$ such that 
\begin{equation}
\label{IneqCriteriaRight}
\sum_{j=1}^{\infty} p_{i,j} x_j \le c x_i,\ i>0,
\end{equation}
iff $c\ge R^{-1}$ (see \cite{Pruitt}). Similarly, 
\begin{equation}
\label{IneqCriteriaLeft}
\sum_{i=1}^{\infty} x_i p_{i,j}\le c x_j,\ j>0,
\end{equation}
iff $c\ge R^{-1}$.

\section{Subcritical Case}
\label{SubcriticalYn}
\subsection{The main result}
Let's start with the classification of subcritical MRSRE. Let $\rho$ be defined by~(\ref{rho}) and $\rho^+=\inf_{h\ge 0} R(h)$. By~(\ref{RAsConvRadius}) $\rho$ equals the minimal $x$ such that
$$
\sum_{n=1}^{\infty} x^{-n} {\bf P}(Y_n>0)<+\infty.
$$
\begin{lem}
\label{LemBetaRho}
Let $\{Y_n\}$ be a subcritical positive zero-absorbing MRSRE.  Assume that for some $\beta\in (0,\widehat{h})$ $R'(\beta)=0$, ${\bf E}^{(\beta)}\xi^2_1<+\infty$, ${\bf E}^{(\beta)}\ln^2 \kappa_1<+\infty$. Then
$\rho \ge \rho^+=R(\beta)$.
\end{lem}
Thus, under the conditions of Lemma~\ref{LemBetaRho} we have two possible situations: $\rho=\rho^+$ (not strongly subcritical) and $\rho>\rho^+$ (strongly subcritical). In BPRE theory the first case is divided into the weakly and intermediate subcritical cases. It seems that in the general case the weakly subcritical case corresponds to 
$$
\sum_{n=1}^{\infty} (\rho^+)^{-n} {\bf P}(Y_n>0) < +\infty
$$
and the intermediate subcritical case corresponds to 
$$
\sum_{n=1}^{\infty} (\rho^+)^{-n} {\bf P}(Y_n>0) = +\infty.
$$
We plan to study these two cases more carefully in future works.

Note that if $\rho^+=0$, then $\rho>\rho^+=0$ by definition and the process is strongly subcritical. Particularly, it's true in the case $\xi=const$. 

By continuity of $R(\cdot)$ there exists $h_{+}^*$ such that 
$$R(h_{+}^*)=\rho,\ R'(h_{+}^*)<0.
$$

The main result of this section is the following.
\begin{thm}
\label{YnSub}
Let $\{Y_n\}$ be a strongly subcritical positive zero-absorbing MRSRE with $\rho\in (0,1)$. Assume that~(\ref{BnCond}) and ${\bf E}^{(h)}\kappa<\infty$ holds for some $h>h^*_+$. Then 
\begin{enumerate}
\item \begin{equation}
\label{YnSubTh}
{\bf P}(Y_n>0|Y_0=i)\sim f^*_i \rho^n,\quad {\bf P}(Y_n=k|Y_n>0)\to p_k^*,\ n\to\infty,
\end{equation}
where $\{p_k^*\}$ is a distribution, and $\{f^*_i\}$ is a bounded sequence such that
$$
\sum_{j=1}^{\infty} p_{i,j} f^*_j=\rho f^*_i,\ i>0,\
\sum_{i=1}^{\infty} p_i^* p_{i,j}=\rho p^*_i,\ j>0,\
\sum_{i=1}^{\infty} p_i^* f^*_i=1;
$$
\item for any $h>h^*_+$ 
$$
{\bf P}(Y_n>0|Y_0=i)\le c i^h \rho^n,\ i>0,
$$
where $c=c(h)$ is some constant.
\item For any positive $h$ such that $R(h)<\rho$ and~(\ref{BnCond}) holds, the distribution $\{p_k^*\}$ has a finite $h$-th moment.
\end{enumerate}
\end{thm}
\begin{remark}
\label{QSDSub}
By Theorem~\ref{QSDForm} under the conditions of~Theorem~\ref{YnSub} 
$$
\pi_i = \frac{1-\rho}{{\bf E} \left(\left.\rho^{1-T_{i}}T_i; T_i<T_0\right|X_0=i\right)}, i\in \mathbb{N},
$$
where $T_i$ is the first return to the state $i$.
\end{remark}
\begin{remark}
\label{TrajectorySub}
Under the conditions of Theorem~\ref{YnSub} for any $0<t_1<t_2<\dotsc<t_k<1$ we have
$$
\lim_{n\to\infty} {\bf P}(Y_{[nt_1]}=i_1,\dotsc, Y_{[nt_k]}=i_k|T>n)\to p_{i_k}^* \prod_{l=1}^{k-1} p_{i_l}^* f_{i_l}^*.
$$
\end{remark}
\begin{remark}
\label{AllTrajectory}
Under the conditions of Theorem~\ref{YnSub} we have
$$
\frac{f^*_{i} {\bf P}(\{Y_j,\ j\le n-1\}\in A, Y_n=i|T>n)}{{\bf P}(\{Y^*_j,\ j\le n-1\}\in A, Y^*_n=i)}\to 1,\ n\to\infty,
$$
uniformly in $A\in \mathbb{N}^{n-1}$, $i\in \mathbb{N}$, where $\{Y^*\}$ is a positively recurrent Markov chain with the transition probabilities $p_{i,j}^*=p_{i,j} f^*_j/(\rho f^*_i)$ and $Y_0=Y_0^*$.
\end{remark}
The main weakness of the theorem is that $\rho$ is unknown and it is hard to compare it with $\rho^+$. We give the following representation of $\rho$. 
\begin{lem}
\label{RhoRepresentation}
i) The process is strongly subcritical iff for some positive $\delta$ and natural $i$
\begin{equation}
\label{ProbPowerSeries}
\sum_{n=1}^{\infty} {\bf P}(Y_n>0|Y_0=i) s^n = +\infty
\end{equation}
for some $s: s\rho^+<1$. Moreover, $1/\rho$ is the radius of convergence of series~(\ref{ProbPowerSeries}).\\
ii) Let $1/\rho^*(h)$ be the radius of convergence of the series
\begin{equation}
\label{ExpPowerSeries}
\sum_{n=1}^{\infty} {\bf E}(Y_n^h|Y_0=i) s^{n},\  h\in (0,\widehat{h}).
\end{equation}
Let $h^*_{++}$ be the solution of the equation
$$
R(h) = \rho,\ R'(h)>0,
$$
where if there's no solution, we say that $h^*_{++}=+\infty$. Then
$$\rho^*(h)=\left\{
\begin{array}{cc}
\rho,& h<h^*_{++},\\
R(h),& h\in [h^*_{++},\widehat{h}).
\end{array}
\right.
$$ 
\end{lem}
\begin{figure}
\begin{center}
\includegraphics[width=250pt]{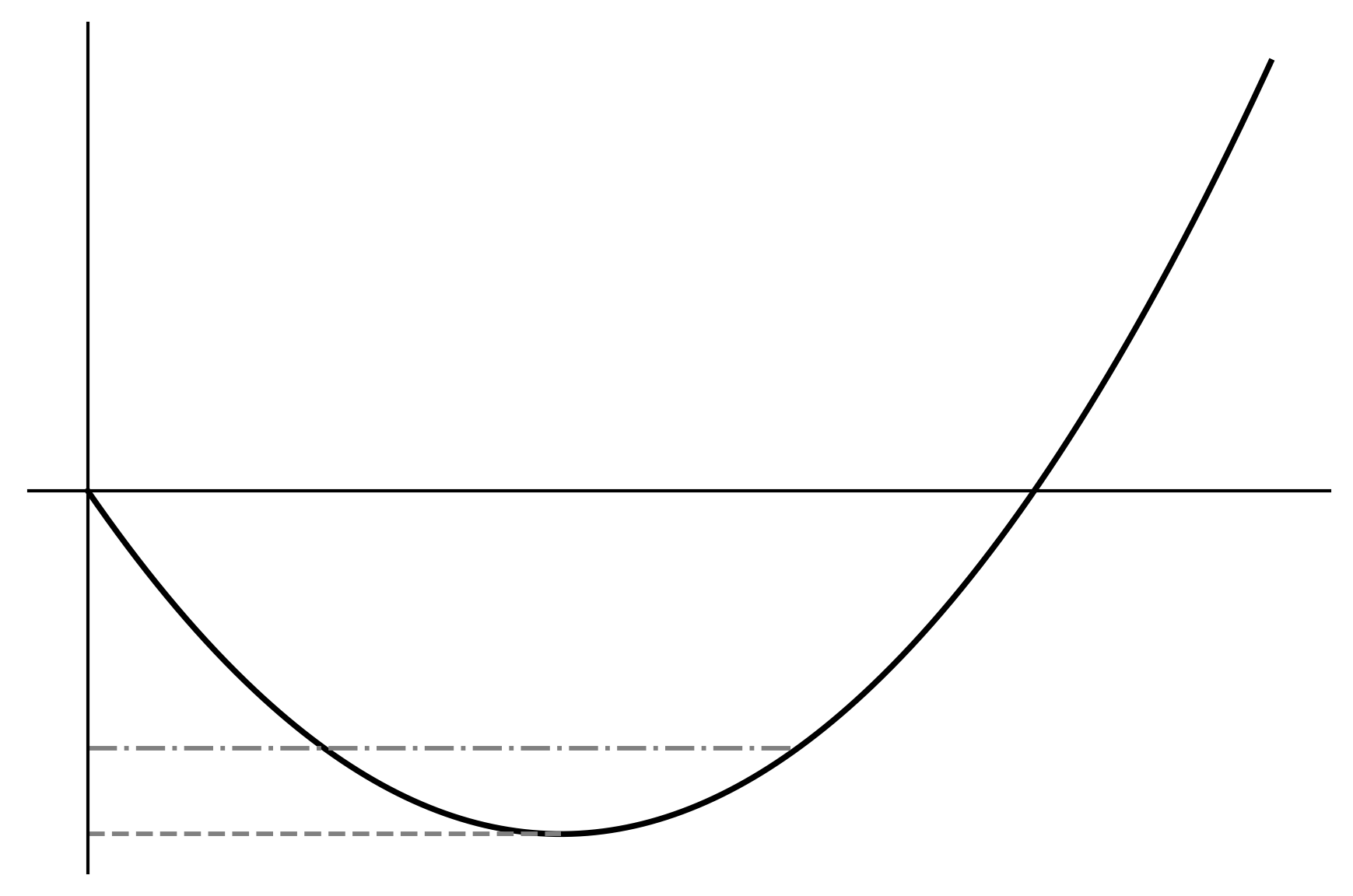}
\caption{The solid line represents the graph of $\ln R(h)$, the dashed line represents $\ln \rho^*(h)$ for $h\le \beta$ in the non-strictly subcritical case (for $h>\beta$ it equals $\ln R(h)$), and the dash-dotted line represents $\ln \rho^*(h)$ for $h\le h^*_{++}$ (for $h>h^*_{++}$ it equals $\ln R(h)$).}
\label{RhoPlot}
\end{center}
\end{figure}

In practice it is usually hard to estimate $\rho$ using series like~(\ref{ProbPowerSeries}) and ~(\ref{ExpPowerSeries}). 
Let $P^+_{M}=(p_{i,j},\ 1\le i,j\le M)$ and $\lambda_M$ be the Perron eigenvalue of $P^+_M$. By Theorem 3.1 of~\cite{SenetaFinite} we can state that $\lambda_M\uparrow \rho$ as $M\to\infty$.  Thus, in practice we can estimate $\rho$ by $\lambda_M$ from below and find M such that $\lambda_M>\rho^+$. Then, the process is strongly subcritical. 
\subsection{Proof of Lemma~\ref{LemBetaRho}}
\begin{proof}
To prove the Lemma it is enough to show that for any positive $\delta$ and some $c$ we have
\begin{equation}
\label{BetaEnough}
{\bf P}(T>n|Y_0=1)\ge c (1-\delta)^n R(\beta)^n, 
\end{equation}
since then
$$
\sum_{n=1}^{\infty} {\bf P}(T>n|Y_0=1) (R(\beta) (1-2\delta))^{-n} = +\infty.
$$
Fix $n$ and note that 
\begin{eqnarray*}
R(\beta)^{-n} {\bf P}(T>n|Y_0=1) = {\bf P}^{(\beta)} \left({\bf P}_{\boldsymbol\eta}(Y_n>0|Y_0=1) e^{-\beta S_n}\right)\ge \\
e^{-\beta n^{2/3}} {\bf P}^{(\beta)}\left(Y_n>0, S_n\le n^{2/3}|Y_0=1\right).
\end{eqnarray*}
For any natural $k$ 
\begin{equation}
\label{YnRec}
Y_n = Y_k e^{S_n-S_k} + \sum_{i=1}^{n-k} B_{k+i} e^{S_n-S_{k+i}}.
\end{equation}
Therefore, for any natural $k$, $l$ we have
\begin{eqnarray}
\label{LowerBoundYn>0}
{\bf P}^{(\beta)}\left(\left.Y_n>0, S_n\le n^{2/3}\right|Y_0=1\right) \ge {\bf P}^{(\beta)}\left(\left.l+\sum_{i=0}^{n-k} B_i e^{-S_i}>0, S_{n-k}\le n^{2/3}/2\right|Y_0=l\right)\times \\ 
{\bf P}^{(\beta)}(Y_k=l, S_k\le n^{2/3}/2)  \ge {\bf E}^{(\beta)}\left({\bf P}_{\boldsymbol\eta}\left(\left.\sum_{i=1}^{n-k} |B_i| e^{-S_i}<l/2\right|Y_0=l\right)\times \right.\nonumber\\
\left.I_{S_{n-k}\le n^{2/3}/2, S_j\ge u_{j}(w), \kappa_j \le v_{j}(w), j\le n-k}\right) {\bf P}^{(\beta)}(Y_k=l, S_k\le n^{2/3}/2),\nonumber
\end{eqnarray}
where for some positive $w$, some $\delta'$ and all $0\le j\le n-k$ we denote
$$
u_{j}:=u_j(w)=w j^{1/2-\delta'},\quad
v_{j}:=v_j(w)=w^{-1} e^{j^{1/2-2\delta'}}.
$$
The probability
$$
{\bf P}_{\boldsymbol\eta}\left(\left.\sum_{i=1}^{n-k} |B_i| e^{-S_i}<l/2\right|Y_0=l\right)
$$
 can be represented in the form
\begin{eqnarray*}
1 - \sum_{t=1}^{n-k-1} {\bf P}_{\boldsymbol\eta}\left(\left.\sum_{i=1}^{j} |B_i| e^{-S_i}<l/2, j<t, \sum_{i=1}^{t} |B_i| e^{-S_i}\ge l/2\right|Y_0=l\right)\ge \\
1 - \sum_{t=1}^{n-k-1} {\bf P}_{\boldsymbol\eta}\left(\sum_{i=1}^{t} |B_i| e^{-S_i}\ge l/2, Y_{i}\in \left(l e^{S_i}/2, 3l e^{S_i}/2\right), i< t\right),
\end{eqnarray*}
since by~(\ref{YnRec})
$$
\left|Y_i e^{-S_i} - Y_0\right|\le \sum_{j=1}^{i} |B_j| e^{-S_j},\ i\ge 0.
$$
By the Markov inequality for any $h\in (0,1)$
\begin{eqnarray}
\label{SumOfBi}
{\bf P}_{\boldsymbol\eta}\left(\sum_{i=1}^{t} |B_i| e^{-S_i}\ge l/2; Y_{i}\in \left(\frac{l}{2} e^{S_i}, \frac{3l}{2} e^{S_i}\right),\ i<t\right)\le \left(\frac 2l\right)^{h}\times \\ \sum_{i=1}^{t} e^{-hS_i} {\bf E}_{\boldsymbol\eta}\left({\bf E}_{\boldsymbol\eta}\left(|B_i|^h|Y_{i-1}\right) ; Y_{i-1}\in \left(\frac{l}{2} e^{S_{i-1}}, \frac{3l}{2} e^{S_{i-1}}\right)\right).\nonumber
\end{eqnarray}
Here we use that 
$$\sum_{i=1}^{t} b_i^h \ge \left(\sum_{i=1}^{t} b_i\right)^h,\ b_i>0,\ h\in (0,1).
$$
By~(\ref{BnCond}) for some $h\in (0,1)$ we have
$$
{\bf E}_{\boldsymbol\eta}\left(|B_i|^h|Y_{i-1}\right)\le 
Y_{i-1}^{\widetilde{h}} e^{\xi_i h} \kappa_i.
$$
Thus, the right-hand side of~(\ref{SumOfBi}) is less than
$$
c l^{\widetilde{h}-h} \sum_{i=1}^{t} e^{-(h-\widetilde{h})S_{i-1}} \kappa_i
$$
for some positive $c$. 
Therefore, by~(\ref{LowerBoundYn>0})
\begin{eqnarray*}
{\bf P}^{(\beta)}\left(\left.Y_n>0, S_n\le n^{2/3}\right|Y_0=1\right) \ge 
 \left(1 - c \sum_{t=1}^{n-k} l^{\widetilde{h}-h}  \sum_{i=1}^{t} e^{-(h-\widetilde{h}) u_{i-1}(w)} v_{i}(w)\right) \times \\
 {\bf P}^{(\beta)}(Y_k=l, S_k\le n^{2/3}/2|Y_0=1)
{\bf P}^{(\beta)}\left(S_n\le n^{2/3}/2, S_j\ge u_{j}(w), \kappa_j \le v_{j}(w)\right).
\end{eqnarray*}
By the definition of $u_{\cdot}$, $v_{\cdot}$ the sequence
$$
 \sum_{t=1}^{\infty} \sum_{i=1}^{t} e^{-(h-\widetilde{h}) u_{i-1}(w)} v_{i}(w)
 $$
 is bounded from above. Therefore, taking $L$ large enough, we obtain for ${\bf P}^{(\beta)}(Y_n>0, S_n\le n^{2/3}|Y_0=1)$, $n\ge k$, the lower bound 
\begin{eqnarray}
\label{BetaRhoBound} 
\frac{1}{2}{\bf P}^{(\beta)}(Y_k=l, S_k\le n^{2/3}/2|Y_0=1) \times\\ {\bf P}^{(\beta)}\left(S_{n-k}\le (n-k)^{2/3}/2, S_j\ge u_{j}(w), \kappa_j \le v_{j}(w),\ j\le n\right|Y_0=1),\ l\ge L.\nonumber
\end{eqnarray}
If we take $k$ large enough, then ${\bf P}(Y_k=l)>0$ for some $l>L$, so ${\bf P}^{(\beta)}(Y_k=l)$ is positive as well. By continuity of measure 
the first probability in~(\ref{BetaRhoBound}) is bounded from below by some positive number for large enough $n$. Thus, we need to show that the second probability in~(\ref{BetaRhoBound}) is not exponentially small in $n$.

Let 
$$
U(x)=I_{x\ge 0} + \sum_{n=1}^{\infty} {\bf P}^{(\beta)}(S_n\ge -x, S_i<0, i\le n).
$$
Note that $U(x)\sim c_1 x$, $x\to+\infty$, for some positive $c_1$ (see~\cite{VatKerst}, Subsection 4.4.3).
Introduce the measure 
$$
{\bf P}^+(A) = {\bf E}^{(\beta)}\left(U(S_n) I_{A}; S_i>0, i\le n\right)
$$
on $\sigma(\eta_1,\eta_2,\dotsc, \eta_n)$. In Section 5.2 of~\cite{VatKerst} it is proved that ${\bf P}^+$ can be extended to $\sigma(\boldsymbol\eta)$. So, 
\begin{eqnarray*}
{\bf P}^{(\beta)}\left(S_{n-k}\le (n-k)^{2/3}/2, S_j\ge u_{j}(w), \kappa_j \le v_{j}(w)\right)\ge \\
\frac{1}{U(n^{2/3})} {\bf P}^{+}\left(S_{n-k}\le (n-k)^{2/3}/2, S_i\ge u_{i}(w), \kappa_i \le v_{i}(w), i>0\right).
\end{eqnarray*}
It is easy to see that
\begin{eqnarray}
\label{P+ineq}
{\bf P}^{+}\left(S_{n-k}\le (n-k)^{2/3}, S_i\ge u_{i}(w), \kappa_i \le v_{i}(w), i>0\right)\ge 1 - {\bf P}^+\left(S_{n-k}>(n-k)^{2/3}\right) - \\
{\bf P}^+\left(\inf \frac{S_i}{i^{1/2-\delta'}}<w\right) -
{\bf P}^+\left(\sup \frac{\zeta_i}{\exp(i^{1/2-2\delta'})}>\frac{1}{w}\right).\nonumber
\end{eqnarray}
The first probability in the right-hand side of~(\ref{P+ineq}) is $o(1)$ as $n\to\infty$ by Theorem 2.1 of \cite{Doney}.
If the random variables
$$
\inf \frac{S_i}{i^{1/2-\delta'}},\ \sup \frac{\zeta_i}{\exp(i^{1/2-2\delta'})}
$$
are ${\bf P}^+$-finite, then for a large enough $w$ the right-hand side of~(\ref{P+ineq}) is greater than 1/2. 
The first variable is finite by (5.17) of~\cite{VatKerst}, the second is finite due to the Borel-Cantelli lemma (see the proof of Lemma 5.5 of~\cite{VatKerst}). 

Thus, the second term in~(\ref{BetaRhoBound}) is greater than $c_2/n^{2/3}$ for some positive $c_2$, and~(\ref{BetaEnough}) is true.
 This proves the lemma.
\end{proof}
\subsection{The proof of Theorem~\ref{YnSub}}
\begin{proof}
i) We only need to check the conditions QS1-QS3. 

Let us start with QS1. Let $K=\{1,2,\dotsc, k\}$ for some $k\in \mathbb{N}$.  Let $\delta<h$ be positive real numbers such that $\widetilde{h}(h)<h-\delta$. Assume that $Y_0=j$ a.s.
Then
$$
{\bf P}(Y_i>k, 1\le i\le n)\le k^{-h} {\bf E}\left(Y_n^{h}I_n\right),
$$
where $I_n = I_{Y_i>k, i\le n}$.
For $h>1$ we have
$$
a_{n}:=\left({\bf E}Y_n^{h}I_n\right)^{1/h}\le \left({\bf E}A_n^h Y_{n-1}^{h}I_{n-1}\right)^{1/h}+
\left({\bf E}(|B_n|^h I_{n-1}\right)^{1/h}.
$$
For $h\in (0,1]$ we have 
$$
a_{n}:={\bf E}Y_n^{h}I_n\le {\bf E}A_n^h Y_{n-1}^{h}I_{n-1}+
{\bf E}|B_n|^h I_{n-1}.
$$
By~(\ref{BnCond}) we have
$$
{\bf E}\left(|B_n|^h I_{n-1}\right)\le {\bf E}\kappa_n e^{h\xi_n} {\bf E} Y_{n-1}^{\widetilde{h}(h)} I_{n-1}\le R(h) k^{-\delta} {\bf E}^{(h)}\kappa_n {\bf E}Y_{n-1}^h I_{n-1}.
$$
Thus, for $h>1$ we have
\begin{equation}
\label{anIneq}
a_n\le a_{n-1} R(h)^{1/h} \left(1+ \left( k^{-\delta}{\bf E}^{(h)}\kappa_n\right)^{1/h}\right).
\end{equation}
For $h\in (0,1]$, similarly,
\begin{equation}
\label{anIneqh<1}
a_n:={\bf E}Y_n^{h}I_n \le a_{n-1} R(h) \left(1+ k^{-\delta}{\bf E}^{(h)}\kappa_n\right).
\end{equation}

Let $h=h^*_++\delta$, where $\delta>0$ is small enough that $R(h)$ is decreasing on $(h^*_+, h)$ and $\widetilde{h}(h)\le h^*_{+}$. Then for $h\ge 1$ and large enough $k$ we have
$$
R(h)^{1/h} \left(1+ \left( k^{-\delta}{\bf E}^{(h)}\kappa_n\right)^{1/h}\right)\le R\left(h-\frac{\delta}{2}\right)^{1/h},
$$
and by~(\ref{anIneq})
$$
a_n\le R\left(h-\frac{\delta}2\right)^{1/h} a_{n-1}\le j \left(1+j^{-\delta} {\bf E}^{(h)}\kappa_1\right)^{1/h} R\left(h-\frac{\delta}2\right)^{n/h}\le C j R\left(h-\frac{\delta}2\right)^{n/h},
$$
where $C$ is some constant. Similarly, for $h\in (0,1]$ we have
$$
a_n \le C j^h R\left(h-\frac{\delta}2\right)^n.
$$

Thus, we have 
\begin{eqnarray*}
{\bf P}(Y_i>k, i\le n|Y_0=j)\le k^{-h} j^h R\left(h^*_++\frac{\delta}2\right) ^{n}.
\end{eqnarray*}
Since $R(h^*_++\delta/2)<R(h^*_{+})$, we obtain the desired estimate. 

To prove QS2, note that for any $i,j\le k$ there exists $l$ such that
$$
p={\bf P}(Y_{l}=j|Y_0=i)>0
$$
Thus,
$$
{\bf P}(T>n|Y_0=i) \ge p {\bf P}(T>n-l|Y_0=j)
$$
and
$$
{\bf P}(T>n|Y_0=j)\le \frac{1}{p} {\bf P}(T>n+l|Y_0=i) \le \frac1p {\bf P}(T>n|Y_0=i).
$$

Finally, QS3 is a simple corollary of irreducibility, since for any $n_0$ and large enough $k_1$
$$
{\bf P}(\exists j\le n_0:\ 1\le Y_j\le k_1|Y_0=i)>0
$$
for all $i\le k$.

Thus, (\ref{YnSubTh}) holds by Theorem~\ref{MainQSTh}. 

ii) Finally, let $\tau_n=\min (j: 0<j\le n: Y_j\le k)$, where if $Y_j>k$ for all $0<j\le n$, then we set $\tau_n=n+1$. Then 
\begin{eqnarray*}
{\bf P}(T>n|Y_0=i) = \sum_{j=1}^{n} \sum_{l=1}^{k} {\bf P}(Y_j=l, \tau_n=j|Y_0=i)
{\bf P}(T>n-j|Y_0=l) + {\bf P}(\tau_n>n|Y_0=i). 
\end{eqnarray*}
From QS1 we know that the second term is less than or equal $c i^h (\rho-\delta)^n$ for some positive $c$, $\delta$. By 
(\ref{YnSubTh}) we have
$$
{\bf P}(T>n|Y_0=l)<c_1 \rho^n 
$$
for some $c_1$ and all $l\in \{1,2, \dotsc, k\}$. On the other hand, 
$$
\sum_{j=1}^{n} {\bf P}(Y_j=l, \tau_n=j|Y_0=i) {\bf P}(T>n-j|Y_0=l)\le c_2 i^h \rho^{n-j} (\rho-\delta)^j
$$
for some positive $c_2$. Thus, for any $i$ we have
$$
{\bf P}(T>n|Y_0=i)\le c_3 \rho^n i^h.
$$

iii) Now, let us prove that $\{p_i^*\}$ has a finite $h$-th moment. We use nice Theorem 2 from~\cite{SenetaMoments} for the  particular case of a finite set $A$. In our case it can be stated as follows: if for some non-negative functions $f(j)$, $g(j)$ and some natural $k\ge 2$ we have 
\begin{equation}
\label{SenetaCond}
\sum_{j=k}^{\infty} p_{i,j} g(j)\le \rho(g(i)-f(i)),\ i\ge k,\quad \sum_{j=k}^{\infty} p_{i,j} g(j)<\infty,\ i<k,
\end{equation}
then 
$$
\sum_{j=1}^{\infty} p_i^* f(i)<+\infty.
$$
Let $g(j)=j^h$. Then for all natural $i$
$$
{\bf E}(g(Y_1)|Y_0=i)^{1/h}\le \left(i^h {\bf E}A_1^h\right)^{1/h}+\left({\bf E}(|B_1|^h|Y_0=i)\right)^{1/h}\le
i R(h)^{1/h} + \left({\bf E}\kappa_1 e^{\xi_1 h}\right)^{1/h} i^{\widetilde{h}/h}.
$$
Thus,
\begin{equation}
\label{Sub-h}
\sum_{j=k}^{\infty} p_{i,j} j^h\le i^h \left(R(h)^{1/h} + c i^{\widetilde{h}/h-1}\right)^h.
\end{equation}
The right-hand side of~(\ref{Sub-h}) has the form
$$i^h R(h) + O\left(i^{h-(h-\widetilde{h})/h}\right),\ i\to\infty.
$$
Let $f(i)=(1-R(h)/\rho-\varepsilon) i^h$, $i>0$. It is positive when $\varepsilon$ is small enough, since $R(h)<\rho$. Then 
$$
i^h \left(R(h)^{1/h} + c i^{\widetilde{h}/h-1}\right)^h\le \rho (g(i)-f(i)) = i^h R(h) + \varepsilon \rho i^h 
$$
for all sufficiently large $i$. Taking $k$ large enough, we obtain the first statement of~(\ref{SenetaCond}). The second follows from~(\ref{Sub-h}).

This finishes the proof of Theorem~\ref{YnSub}.

Finally, Remark~\ref{TrajectorySub} is a corollary of Remark 1 of~\cite{Quasi}. Let us prove Remark~\ref{AllTrajectory}. Let $Y_0=i_0$ a.s. Note that for any $\{i_j, 1\le j\le n\}$ we have
\begin{eqnarray*}
{\bf P}(Y_j=i_j, j\le n|T>n)=\frac{1}{{\bf P}(T>n)} \prod_{j=1}^{n} p_{i_{j-1}, i_j} = 
\frac{f_{i_0} \rho^n}{f_{i_n} {\bf P}(T>n|Y_0=i_0)} 
 \prod_{j=1}^{n} \frac{p_{i_{j-1}, i_j} f_{i_j}}{\rho f_{i_{j-1}}}.
\end{eqnarray*}
Thus, for any $A\in \mathcal{B}(\mathbb{Z}^n)$ we have
$$
{\bf P}(\{Y_j,\ j\le n\}\in A|T>n)= \frac{f_{i_0} \rho^n}{{\bf P}(T>n|Y_0=i_0)}  {\bf E}\frac{1}{f_{Y_n^*}} I_{\{Y_j^*,\ j\le n\}\in A},
$$
Therefore, Remark~\ref{AllTrajectory} is a corollary of Theorem~\ref{YnSub}. The chain $\{Y_n^*\}$ is a positively recurrent due to~\cite{Quasi}. 
\end{proof}
\subsection{Proof of Lemma~\ref{RhoRepresentation}}
\begin{proof}
Note that ${\bf E}u^T<+\infty$, $u>1$, iff 
$$
\sum_{n=1}^{\infty} u^n {\bf P}(T>n)<+\infty.
$$
Thus,~(\ref{RAsConvRadius}) is equivalent to the first statement of the lemma. 

Fix some natural $i$ and positive $h$. Let us prove that $\rho\le \rho^*(h)\le \max(R(h), \rho)$, $h\in [0,\widehat{h})$, $\rho\ge R(h)$, $h>\beta$. 

Note that if the series~(\ref{ExpPowerSeries}) converges for some $s\rho^+<1$, then 
$$
\sum_{n=1}^{\infty} s^n {\bf P}(Y_n>0|Y_0=i)<+\infty,
$$
since ${\bf P}(Y_n>0|Y_0=i)\le {\bf E}(Y_n^h|Y_0=i)$. Thus, $\rho\le \rho^*(h)$.

On the other hand, suppose that $\rho<\rho^*$. For any positive $\delta$ we have
\begin{eqnarray*}
\sum_{n=1}^{\infty} \exp(\delta h n) (\rho^*)^{-n}{\bf E}(Y_n^h|Y_0=i)=+\infty.
\end{eqnarray*}
Fix some $k$ and note that 
$$
{\bf E}(Y_n^h|Y_0=i)\le \sum_{m=0}^{n}\sum_{j=1}^{k} {\bf E}\left(Y_{n-m}^h I_{n-m}|Y_0=j\right) {\bf P}(Y_{m}=j|Y_0=i) .
$$
By the definition of $\rho$, for some $c$, all natural $m$, and $1\le j\le k$ we have
$$
{\bf P}(Y_{m}=j|Y_0=i)\le c \rho^{m} e^{\delta h m}.
$$
By~(\ref{anIneq}) for $h>1$, some $c_1$, and any $\delta_1\in (0,h-\widetilde{h}(h))$ we have
$$
{\bf E}\left(Y_{n-m}^h I_{n-m}|Y_0=j\right) \le R(h)^{n-m} \left(1+\frac{c_1}{k^{\delta_1/h}}\right)^{(n-m) h}.
$$
Therefore, for some constant $C$ we have
\begin{eqnarray*}
{\bf E}(Y_n^h|Y_0=i) \le C n k \max\left(\rho e^{\delta h}, R(h) \left(1+\frac{c_1}{k^{\delta_1/h}}\right)^h\right)^n.
\end{eqnarray*}
The case $h\in (0,1]$  is considered similarly.
Taking $k$ large enough and $\delta$ small enough, we get $\rho\le \rho^*(h)\le \max(\rho, R(h))$. 

If $h^*_{++}$ is finite, then $\beta$ is finite too. Let us prove that if $h>\beta$, then $\rho^*(h)\ge R(h)$ and $\rho^*=\max(\rho, R(h))=R(h)$ for $h>h^*_{++}$. Fix some $c$ and note that 
$$
{\bf E}\left(\left.Y_n^h\right|Y_0=i\right)\ge R(h)^n {\bf E}^{(h)}\left(\left.\left(\frac{Y_{n}}{e^{S_n}}\right)^h I_{J_0}\right|Y_0=i\right),\quad J_0:= \left\{S_j\ge m(h) j/2,\ \kappa_j\le c j\right\}.
$$
Let $\delta\in (0, (1- \widetilde{h}(h)/h)/2)$. For any natural $j$, $l$ we have
\begin{equation}
\label{BjIneqLemma2}
{\bf P}\left(\left. |B_j|>l^{1-\delta} e^{\xi_j} \right|Y_{j-1}=l\right)\le \frac{\kappa_j l^{\widetilde{h}(h)}}{l^{(1-\delta)h}}\le \frac{\kappa_j}{l^{\delta h}}. 
\end{equation}
Let 
$$
d_n = d_{n-1}- d_{n-1}^{1-\delta} e^{-\delta S_{n-1}},\ n\ge 1,\ d_0=i.
$$

If $i$ is large enough that 
$$
 i \left(1- i^{-\delta}\right) \prod_{j=1}^{\infty} \left(1-e^{-\delta m(h) (j-1)/2}\right)\ge 2,
$$
then on $J_0$ we have $d_n\ge 1$. We prove it by induction on $n$. First, $d_1 = i (1-i^{-\delta})\ge 1$. If $d_j\ge 1$, $j\le n-1$, then
$$
d_n \ge d_{n-1} \left(1 - e^{-\delta m(h) (n-1)}\right) \ge i \prod_{j=1}^{n} \left(1-e^{-\delta m(h) (j-1)/2}\right)\ge 1.
$$

Let
$$
J_{j,1} = \left\{|B_j|\le Y_j^{1-\delta} e^{\xi_j}\right\},\ 
J_{j,2} = \left\{Y_j\ge d_j e^{S_j}\right\},\
j\ge 1.
$$ 
Note that 
$$J_{n,2}\cap J_0\subseteq \left\{Y_n\ge e^{S_n}\right\}.$$
Then 
\begin{eqnarray*}
{\bf P}_{\boldsymbol\eta}\left(J_{j,1}, J_{j,2},\ j\le n\right)\ge 
{\bf P}_{\boldsymbol\eta}\left(J_{j,1}, J_{j,2},\ j\le n-1\right) 
\min_{l\ge d_{n-1} e^{S_{n-1}}} {\bf P}_{\boldsymbol\eta}\left(\left.J_{n,1}\right|Y_{n-1}=l\right).
\end{eqnarray*}
Here we used the fact that if $J_{n,1}$ and $J_{n-1,2}$ are true, then
$$
Y_n \ge Y_{n-1} e^{\xi_n} - Y_{n-1}^{1-\delta} e^{\xi_n} = e^{\xi_n} (Y_{n-1}-Y_{n-1}^{1-\delta})\ge e^{S_n}\left(\frac{Y_{n-1}}{e^{S_{n-1}}} - e^{-\delta S_{n-1}} \frac{Y_{n-1}}{e^{S_{n-1}}}\right) \ge e^{S_n} d_{n}.
$$
Thus, by~(\ref{BjIneqLemma2}) on $J_0$ we have
$$
{\bf P}_{\boldsymbol\eta}\left(J_{j,1}, J_{j,2},\ j\le n\right)\ge \prod_{j=1}^{n} \left(1-\frac{\kappa_j}{d_j^{\delta h} e^{\delta h S_j}}\right)\ge  \prod_{j=1}^{\infty} \left(1-\frac{c j}{e^{\delta h m(h) j/2}}\right)>\delta_1,
$$
where $\delta_1$ is some positive number. Therefore,
$$
{\bf E}\left(\left.Y_n^h\right|Y_0=i\right)\ge R(h)^n {\bf E}^{(h)}\left({\bf P}_{\boldsymbol\eta}\left(J_{j,1}, J_{j,2},\ j\le n\right) I_{J_0}\right)\ge R(h)^n \delta_1 {\bf P}^{(h)}(J_0).
$$
But
$$
{\bf P}^{(h)}(J_0)\ge {\bf P}^{(h)}\left(S_j - \frac{m(h) j}{2} \ge 0, j\ge 0\right) - {\bf P}\left(\max_j \frac{\kappa_j}{j}>c\right).
$$
The first probability is positive, since $S_j-m(h)j/2$ is a random walk with positive mean. Since ${\bf E}\kappa_1<+\infty$ and $\kappa_j$ are i.i.d., we have $\kappa_j=O(j)$, $j\to \infty$ (by Lemma 3 of Paragraph 3 of Chapter 4 of \cite{Shiryaev} and Borel-Cantelli Lemma). Taking $c$ large enough, we get 
$$
{\bf P}^{(h)}(J_0) \ge \frac12{\bf P}^{(h)}\left(S_j - \frac{m(h) j}{2} \ge 0, j\ge 0\right) > 0.
$$
Therefore, $\rho^*(h)\ge R(h)$, $h>\beta$. This finishes the proof.
\end{proof}
\section{MRSRE with no absorption}
\label{NoAbsorption}
Assume that $\{Y_n\}$ is an irreducible aperiodic MRSRE (including $\{0\}$). Then we can state the following theorem.
\begin{thm}
\label{NoAbsorptionSubTh}
Let $\{Y_n\}$ be a subcritical irreducible aperiodic MRSRE. Assume that~(\ref{BnCond}) holds for some $h>0$, ${\bf E}^{(h)}\kappa_1<+\infty$, and  ${\bf E}\left(\left.Y_1^h\right|Y_0=0\right)<+\infty$. Then $\{Y_n\}$ is geometric ergodic, so
$$
{\bf P}(Y_n=k)\to \pi_k,\ n\to\infty,\ k\ge 1,
$$
where $\{\pi_k\}$ is the unique stationary distribution and the convergence is exponentially fast.
\end{thm}
\begin{proof}[Proof of Theorem~\ref{NoAbsorptionSubTh}]
Let $T_0$ be the time until the first return to $\{0,1,2,\dotsc, k\}$ from $i\in \mathbb{N}$. For $h\in (0,1)$
$$
{\bf P}(T_0>n|Y_0=i)={\bf P}(Y_j>k, j\le n|Y_0=i)\le k^{-h} {\bf E}(Y_n^h I_n|Y_0=i)=:k^{-h} a_{n},
$$
and
$$
a_n\le a_{n-1} R(h) \left(1 + k^{-\delta} {\bf E}^{(h)} \kappa_n \right).
$$
Assume that $h$ satisfies $R(h)<1$. Then the parameter $k$ can be chosen large enough so that 
$$
a_n\le c i^h  (R(h)+\delta)^n
$$
for some $\delta<1-R(h)$. Thus, for any $i\in \{1,2,\dotsc,k\}$ the probability ${\bf P}(T_0>n|Y_0=i)$ is exponentially small. If $i=0$, then
$$
{\bf P}(T_0>n|Y_0=0)\le \sum_{i=k+1}^{\infty} {\bf P}(T_0>n-1|Y_0=i) {\bf P}(Y_1=i|Y_0=0)\le c {\bf E}(Y_1^h|Y_0=0) (R(h)+\delta)^n.
$$
Therefore, ${\bf P}(T_0>n|Y_0=i)$ is exponentially small for all $i\le k$. Thus,
$$
\sup_{0\le i\le k} {\bf E}\left(\left.e^{\delta_1 T_0}\right|Y_0=i\right)<+\infty
$$
for some $\delta_1>0$. By Theorem 15.0.1 of \cite{GeometricErg} the chain is geometrically ergodic (it is obvious that $\{0,1,\dotsc, k\}$ is a small set).
\end{proof}
\section{Applications}
\label{Applications}
We want to apply the general theory to particular models of branching processes. The key question is how to find or estimate $\rho$. We can  use one of following approaches:
\begin{enumerate}
\item to describe $\rho$ using the definition~(\ref{rho});
\item to find $\rho$, using~(\ref{RAsConvRadius});
\item to study one of the equations $\pi P_{+}=\rho \pi$ or $P^+ f = \rho f$ (and use~(\ref{IneqCriteriaRight}) or (\ref{IneqCriteriaLeft}));
\item to use equation~(\ref{ExpPowerSeries});
\item in the case of models in a random environment we can use the quenched approach: to study inhomogeneous Markov chain $\{Y_n\}$ for a fixed environment;
\item numerically we can also use the approach, described below lemma~\ref{RhoRepresentation}.
\end{enumerate}
\subsection{Branching processes and branching processes in a random environment}
Let $\boldsymbol\eta$ be the environment. For fixed $\boldsymbol\eta$ let $X_{i,j}\sim P_{\eta_i}$, $i,j\ge 1$, be independent r.v., where $\{P_y,\ y\in\mathbb{R}\}$ is some distribution family. The branching process in a random environment (BPRE) is defined as follows:
$$
Z_0=1,\quad Z_{n+1} = \sum_{i=1}^{Z_n} X_{n+1,i},\ n\ge 1.
$$
Let $\xi_i = \ln {\bf E}_{\boldsymbol\eta} X_{i,1}$, $i\in \mathbb{N}$. Note that we assume that ${\bf P}_{\boldsymbol\eta} (X_{1,1}=0)<1$ for a.s. $\boldsymbol\eta$.

This first work about BPRE were~\cite{Smith1969},~\cite{Athreya-1}, and~\cite{Athreya-2}. For a brief review see \cite{VatKerst}. 

Then BPRE $\{Z_n\}$ is a MRSRE with 
$$
A_{n} = e^{\xi_n},\quad B_n = Z_{n} - A_n Z_{n-1}=\sum_{i=1}^{Z_{n-1}} \left(X_{n,i}-{\bf E}_{\boldsymbol\eta} X_{n,1}\right).
$$
If ${\bf E}X^h$ is finite for some $h>1$, then~(\ref{BnCond}) is a corollary of Theorem~\ref{BahrTh}.
$$
{\bf E}_{\boldsymbol\eta} \left(\left.\left|\sum_{i=1}^{Z_{n-1}} \left(X_{n,i}-{\bf E}_{\boldsymbol\eta} X_{n,1}\right)\right|^{h}\right|Z_{n-1}\right)\le c {\bf E}|X_{n,i}-{\bf E}_{\boldsymbol\eta} X_{n,1}|^h Z_{n-1}^{\widetilde{h}},
$$
where $c$ is some constant and $\widetilde{h}(h)=\min(h/2,1)$.

So, we can easily apply Theorem~\ref{YnSub} to this model. Let describe the parameter $\rho$ and the limit distribution $\{p_i^*\}$. This problem is well-known (see \cite{VatKerst}, Theorem 9.1).
However, we use this classic example to illustrate the approaches i-v. We just sketch the main ideas of this approaches.

i) For ordinary BP this method was applied in~\cite{Seneta-Vere-Jones}, Section 5. For BPRE the same approach can be used. Note that
\begin{eqnarray*}
\sqrt[n]{{\bf P}(Z_n=1|Z_0=1)} = \sqrt[n]{{\bf E}\prod_{i=1}^{n} \phi_{\eta_i}'(\phi_{\eta_{i+1}}(\dotsc \phi_{\eta_n}(0)\dotsc))} = \\
R(1) \sqrt[n]{{\bf E}^{(1)}\prod_{i=1}^{n} \frac{\phi_{\eta_i}'(\phi_{\eta_{i+1}}(\dotsc \phi_{\eta_n}(0)\dotsc))}{\phi_{\eta_i}'(1)}}\le R(1). 
\end{eqnarray*}
Thus, $\rho\le R(1)$. On the other hand, if $R'(1)<0$ and ${\bf E}X_1 \ln X_1<+\infty$, then $\rho=R(1)$ by (9.6) of \cite{VatKerst}. 

ii) Since 
$${\bf P}(T>n) = {\bf E}\left(1-\phi_{\eta_1}(\phi_{\eta_2}(\dotsc \phi_{\eta_n}(0))\right)=
\prod_{i=1}^{n} \phi'_{\eta_i}(s_{i+1}),\ s_{i+1}\in [\phi_{\eta_{i+1}}(\phi_{\eta_{i+2}}(\dotsc \phi_{\eta_n}(0)  \dotsc)),1].
$$
So, $\rho = R(1)$ in the case $R'(1)<0$ by the same arguing as above.

iii) The equation $\pi P^+ = \rho \pi$ is equivalent to 
$$
\rho \psi(s) = \sum_{j=1}^{\infty} \sum_{i=1}^{\infty} \pi_i p_{i,j} s^j = \sum_{i=1}^{\infty} \pi_i {\bf E}\left(\varphi_1^i(s) - \varphi_1^i(0)\right) = {\bf E}\psi(\varphi_{\eta_1}(s))-{\bf E}\psi(\varphi_{\eta_1}(0)),
$$
where 
$\psi(s)$ is a p.g.f. of $\pi$. Let
$$
(Af)(s) = {\bf E}f(\varphi_{\eta_1}(s)) - {\bf E}f(\varphi_{\eta_1}(0))
$$
be a linear operator on $C[0,1]$. Then $\rho$ is an eigenvalue and $\psi$ is an eigenfunction of $A$. The space of non-negative functions is a positive cone in $C[0,1]$, thus, it's natural to use analogues of Perron-Frobenius for Banach spaces (see
\cite{Marek}) to find $f$ and $\rho$. It looks like $\rho$ is a Perron eigenvalue and $f$ is the corresponding function of $A$, but we fail to prove it.

iv) Suppose that $R'(\beta)=0$ for some $\beta$. We know that
$${\bf E}Z_n = R(1)^{n},
$$
thus $\rho^*(1)=R(1)$. If $R'(1)<0$, then $\rho^*(1+\delta)>R(1+\delta)$ for some positive $\delta$ and we can apply Lemma~\ref{RhoRepresentation} to say that $\rho=\rho^*(h)$ for all $h\le 1+\delta$. So, $\rho=R(1)$. 

v) If we fix the environment, we can use the quenched approach. If ${\bf P}(\eta_n)$ is the conditional transition matrix of BPRE in the moment $n$, then $f(i) = i$ satisfies
$$
\sum_{j=1}^{\infty} p_{i,j}(\eta_n) f(j) = e^{\xi_n} f(j). 
$$ 
Thus,
$$
q_{i,j}(\eta_n):=\frac{p_{i,j}(\eta_n) j}{e^{\xi_n} i}
$$
is the new family of transition matrices. Let $\{U_n\}$ be the corresponding non-homogeneous Markov chain. Then 
$$
{\bf P}(Z_n=j|Z_0=i) = \frac{i}{j} {\bf E}e^{\xi_1+\xi_2+\dotsb+\xi_n} {\bf P}_{\boldsymbol\eta}(U_n=j|U_0=i) = \frac{i}{j} R(1)^n {\bf P}(V_n=j|V_0=i),
$$
where $\{V_n\}$ is new Markov chain on $\mathbb{N}$ with
$$
{\bf P}(V_1=j|V_0=i) = {\bf E}^{(1)} {\bf P}_{\boldsymbol\eta}(U_1=j|U_0=i),\ i,j\in \mathbb{N}.
$$
One can see that $\{U_n\}$ is a branching process in a varying environment, where all the particles except one in the $i$-th generation give a random number of descendants $X_{i,\cdot}\sim \varphi_{\eta_i}$, and this particle gives a random number of descendants $\widetilde{X}_i$ with the p.g.f. $\phi_{\eta_i}'(s)/\phi_{\eta_i}'(1)$. This is similar to the Geiger tree construction, see Section 1.4 of \cite{VatKerst}. If one show that $\{V_n\}$ is ergodic as $R'(1)<0$, then $\rho=R(1)$. 
By Theorem~\ref{BahrTh} for any positive $h$ 
\begin{eqnarray*}
{\bf E}_{\boldsymbol\eta}\left(\left.\left|V_n-e^{\xi_n} V_{n-1}\right|^h\right|V_{n-1}=j\right) =\left({\bf E}_{\boldsymbol\eta}\left|\widetilde{X}_{n}-e^{\xi_n}+\sum_{i=1}^{j-1} \left(X_{n,i} - e^{\xi_n}\right)\right|^{1+\delta}\right)^{h/(1+\delta)}\le\\ 
\left({\bf E}_{\boldsymbol\eta}\left|\widetilde{X}_{n}-e^{\xi_n}\right|^{1+\delta}+2(j-1){\bf E}_{\boldsymbol\eta}\left|X_{n,1} - e^{\xi_n}\right|^{1+\delta}\right)^{h/(1+\delta)}\le \kappa_n e^{h \xi_n} j^{h/(1+\delta)},
\end{eqnarray*}
where
$$
\kappa_n = e^{-h \xi_n} \max\left({\bf E}_{\boldsymbol\eta}\left|\widetilde{X}_{n}-e^{\xi_n}\right|^{1+\delta},2{\bf E}_{\boldsymbol\eta}\left|X_{n,1} - e^{\xi_n}\right|^{1+\delta}\right)^{h/(1+\delta)}.
$$
Since
$$
{\bf E}^{(1+h)} \kappa_1 = {\bf E}\left(e^{\xi_n}\left({\bf E}_{\boldsymbol\eta} \widetilde{X}_1^{1+\delta}\right)^{h/(1+\delta)}\right) +   
{\bf E}\left(e^{\xi_1}\left({\bf E}_{\boldsymbol\eta} \widetilde{X}_{1,1}^{1+\delta}\right)^{h/(1+\delta)}\right),
$$
if the expectations in the right-hand side are finite for some positive $h$, then by Theorem~\ref{NoAbsorptionSubTh} the sequence $\{V_n\}$ is ergodic and $\rho=R(1)$.

Thus, different ways lead us to the fact that if $R'(1)<0$, ${\bf E}X^{1+\delta}<+\infty$ for some $\delta$, then
$$
{\bf P}(Z_n>0|Z_0=j)\sim f_j R(1)^n,\quad {\bf P}(Z_n=i|Z_0=j)\sim p_i^* R(1)^n,\ n\to\infty.
$$
 This result is well known. Our conditions are more restrictive then the classical ones since we work only in geometrically ergodic case (in sense of quasi-stationary distribution).
\subsection{Branching processes with migration and branching process with migration in a random environment}
Let $\boldsymbol\eta$ be the environment. For fixed $\boldsymbol\eta$ let $X_{i,j}\sim {\bf P}_{\eta_i}$, $\zeta_i\sim {\bf \widetilde{P}}_{\eta_i}$, $i,j\ge 1$, be independent integer-valued r.v., where $\{{\bf P}_y, {\bf \widetilde{P}}_y,\ y\in\mathbb{R}\}$ is some distribution family. Let $\phi_y$, $\widetilde{\phi}_y$ be the corresponding p.g.f's of ${\bf P}_y$ and ${\bf \widetilde{P}}_y$. The branching process with migration in a random environment (BPMRE) is defined as follows:
$$
Z^*_0=1,\quad Z^*_{n+1} = \max\left(0, \sum_{i=1}^{Z^*_n} X_{n+1,i}+\zeta_{n+1}\right),\ n\ge 1.
$$
Let $\xi_i = \ln {\bf E}_{\boldsymbol\eta} X_{i,1}$, $i\in \mathbb{N}$. Note that we assume that ${\bf P}_{\boldsymbol\eta} (X_{1,1}=0)<1$ for a.s. $\boldsymbol\eta$.

Notice that $X_{i,j}$ are supposed to be non-negative, but we allow $\zeta_{\cdot}$ to be negative. This corresponds to emigration. The process with $\zeta_{1}\ge 0$ a.s. is called the branching process with immigration in a random environment (BPIRE), if $\zeta_{1}\le 0$ we call the process the branching process with emigration in a random environment (BPERE).

Similarly, we can can consider a branching process with migration stopped at zero in a random environment (BPMRE0) $Z_n^{0}$. The process is defined in the same way, but zero is set to be absorbing.

Both BPMRE and BPMRE $\{Z^*_n\}$ are MRSRE with 
$$
A_{n} = e^{\xi_n},\quad B_n = \max\left(\sum_{i=1}^{Z^*_{n-1}} \left(X_{n,i}-{\bf E}_{\boldsymbol\eta} X_{n,1}\right)+\zeta_{n+1}, -A_n Z^*_{n-1}\right).
$$

If ${\bf E}X^h$, ${\bf E}|\zeta|^h$ are finite for some $h>1$, then~(\ref{BnCond}) is a corollary of Minkowski inequality 
\begin{eqnarray*}
{\bf E}_{\boldsymbol\eta}\left(\left.|B_n|^h\right|Z^*_{n-1}\right)\le \left(\left({\bf E}_{\boldsymbol\eta}\left(\left.\left|\sum_{i=1}^{Z^*_{n-1}} \left(X_{n,i}-{\bf E}_{\boldsymbol\eta} X_{n,1}\right)\right|^h\right|Z^*_{n-1}\right)\right)^{1/h} + \left({\bf E}_{\boldsymbol\eta}|\zeta_n|^h\right)^{1/h}\right)^h\le \\
\left(c \left({\bf E}_{\boldsymbol\eta}|X_{n,i}-{\bf E}_{\boldsymbol\eta} X_{n,1}|^{h}\right)^{1/h} + \left({\bf E}_{\boldsymbol\eta}|\zeta_n|^h\right)^{1/h}\right)^h \left(Z_{n-1}^*\right)^{\max(h/2,1)}.
\end{eqnarray*}
So, $\widetilde{h}(h)=\max(h/2,1)$. 

Note that for $0<h\le 1$ we can instead assume that for some $\delta\in (0,1)$
\begin{equation}
\label{h<1}
{\bf E}\left({\bf E}_{\boldsymbol\eta}X_{n,1}^{1+\delta}\right)^{h/(1+\delta)}<+\infty,\quad {\bf E}|\zeta_1|^h<+\infty.
\end{equation}
Really, 
\begin{eqnarray*}
{\bf E}_{\boldsymbol\eta}\left(\left.|B_n|^h\right|Z^*_{n-1}\right)\le {\bf E}_{\boldsymbol\eta}\left(\left.\left|\sum_{i=1}^{Z^*_{n-1}} \left(X_{n,i}-{\bf E}_{\boldsymbol\eta} X_{n,1}\right)\right|^{1+\delta}\right|Z^*_{n-1}\right)^{h/(1+\delta)} + {\bf E}_{\boldsymbol\eta}|\zeta_n|^h\le \\
\left(2 \left({\bf E}_{\boldsymbol\eta}|X_{n,i}-{\bf E}_{\boldsymbol\eta} X_{n,1}|^{1+\delta}\right)^{h/(1+\delta)} + {\bf E}_{\boldsymbol\eta}|\zeta_n|^h\right) \left(Z_{n-1}^*\right)^{h/(1+\delta)}
.
\end{eqnarray*}
So,~(\ref{BnCond}) holds for some $\widetilde{h}(h)=h/(1+\delta)$. The same facts are true for BPMRE0, since~(\ref{BnCond}) is stated only for non-absorbing states. 

Consider the subcritical BPMRE (${\bf E}\xi<0$) and assume that the chain $\{Z_n^*\}$ is irreducible. Suppose that condition (\ref{h<1}) holds for some $t>1$ and ${\bf E}|\zeta|^h$ is finite for some $h>0$. Then the process satisfies the conditions of Theorem~\ref{NoAbsorptionSubTh}. Thus, the process $\{Z_n^*\}$ is geometrically ergodic. In recent work~\cite{Kevei} the ergodic distribution of BPIRE was studied more carefully under weaker conditions.

Consider an irreducible subcritical BPMRE0 $\{Z_n^*\}$. Assume that $\rho>\rho^+$ and ${\bf E}X^h<\infty$, ${\bf E}|\zeta|^{h}$ for some $h>\max(1,h^*_{+})$ or conditions~(\ref{h<1}) hold for some $1>h>h^*_{+}$ and some $t>1$.  Then for any $i>0$ we have
$$
{\bf P}(Z_n^*>0|Z_0=i)\sim f_i \rho^n,\quad 
{\bf P}(Z_n^*=i|Z_n^*>0)\to p_i^*,\ n\to\infty.
$$
It is hard to find the parameter $\rho$ in this case. First, see that for BPIRE0 we have $\rho\ge R(1)$, for BPERE0 we have $\rho\le R(1)$, since the corresponding processes are stochastically greater (less) then BPRE. 

In~\cite{SubcritIm} similar results for BPRIRE0 were obtained under weaker conditions. The parameter $\rho$ was represented as the smallest positive $\rho$ such that 
$$
\sum_{n=1}^{\infty} \rho^{-n} {\bf P}(T>n)=+\infty.
$$
This corresponds to the approach ii): $\rho$ can be found as the smallest solution of ${\bf E}\rho^{-T}=+\infty$. 

Notice, that the classification in~\cite{SubcritIm} is the same as for BPRE and is based on $R'(1)$. It looks like this classification is not natural for BPIRE -- if $R'(h_+^*)<0<R'(1)$, then the process is called ''weakly subcritical'', but the asymptotics is the same as for ''strongly subcritical'' BPIRE (this case corresponds to the first part of Theorem 1 in~\cite{SubcritIm}). We suppose that the case ii of Theorem 1 of~\cite{SubcritIm} corresponds to the intermediate subcritical behavior and the case iii -- to the weakly subcritical behavior.  

It looks like the approaches iv) and v) are hard to apply to this model, but the approach iii) gives a new look. Similarly as for BPRE if
$$
(Af)(s) = {\bf E}g(\varphi_{\eta_1}(s)) \widetilde{\varphi}_{\eta_1}(s) - {\bf E}g(\varphi_{\eta_1}(0)) \widetilde{\varphi}_{\eta_1}(0),\ f\in C[0,1],
$$
then $\rho$ is a positive eigenvalue and p.g.f. of $p_i^*$ is a positive eigenfunction of the operator $A$. 
\subsection{Bisexual Branching Processes and Bisexual Branching Processes in a Random Environment}
Let $((W_{n,i},F_{n,i}), i\ge 1)$ be independent sequences of i.i.d. random vectors with some distribution $\mathcal{P}$ on $\mathbb{N}_0\times \mathbb{N}_0$. Assume that $L:\mathbb{N}_0\times \mathbb{N}_0\to \mathbb{N}_0$ is some function, called mating. Then bisexual branching process (BBP) is defined by the relation
$$
N_{n+1} = L\left(\sum_{i=1}^{N_n} W_{n+1,i}, \sum_{i=1}^{N_n} M_{n+1,i}\right),\ n\ge 1,\ N_0=1.
$$
This model has a natural biological interpretation. Assume that $\{N_n\}$ is a number of families. Then $i$-th family of $n$-th generation produces a random number $M_{n,i}$ of male descendants and a random number $F_{n,i}$ of female descendants. After that the total number of $x$ males and $y$ females forms $L(x,y)$ families. 

The BBP in a random environment (BBPRE) is defined similarly. We fix a random sequence $\boldsymbol\eta$ of i.i.d. r.v. Assume that for any $n$, $i$ the random vector $(W_{n,i},F_{n,i})$ has conditional distribution $\mathcal{P}_{\eta_n}$ and $(W_{n,i},F_{n,i})$ are conditionally independent for given $\boldsymbol\eta$, $i,n\in\mathbb{N}$. We define the process by
$$
N_{n+1} = L\left(\sum_{i=1}^{N_n} W_{n+1,i}, \sum_{i=1}^{N_n} M_{n+1,i}, \eta_{n+1}\right),\ n\ge 1,\ N_0=1,
$$
where for every $z\in \mathbb{R}$ $L(x,y,z)$ is a mating function.

We always assume that the mating function $L(x,y,z)$ satisfy the key appoximation condition: there exists function $g:\mathbb{R}_0\times \mathbb{R}_0\to \mathbb{R}_0$, function $c_1:\mathbb{R}\to \mathbb{R}^+$ and positive constant $\delta$ such that
\begin{equation}
\label{ApproxLipBBP}
|L(x,y,z) - g(x,y,z)|\le c_1(z) (x+y)^{1-\delta}.
\end{equation}
Moreover, we assume that $g(x,y,z)$ is Lipschitz in the first two arguments with the constant $c_2(z)$ and $g(cx, cy, z) = c g(x,y,z)$, $x,y,c\in\mathbb{R}_0$. 

Then (see \cite{ShkBBPRE} for details) $\{N_n\}$ can be represented as MRSRE:
$$
N_{n+1} = A_{n+1} N_n + B_{n+1},\ n\ge 0,
$$
where
\begin{eqnarray*}
A_{n}=g\left({\bf E}_{\boldsymbol\eta} F_{n,1}, {\bf E}_{\boldsymbol\eta} M_{n,1}
\right),\ n\ge 1.
\end{eqnarray*}
The process is subcritical if ${\bf E}\ln A_n<0$. For $h \ge (1-\delta)^{-1}$ we have
\begin{eqnarray*}
{\bf E}_{\boldsymbol\eta}\left(\left.|B_{n+1}|^h\right|N_n=k\right)^{1/h} \le\\
\left({\bf E}_{\boldsymbol\eta}\left|L\left(\sum_{i=1}^{k} F_{n+1,i},  \sum_{i=1}^{k} M_{n+1,i}, \eta_{n+1}\right) - 
g\left(\sum_{i=1}^{k} F_{n+1,i},  \sum_{i=1}^{k} M_{n+1,i}, \eta_{n+1}\right) \right|^h\right)^{1/h} + \\
\left({\bf E}\left|g\left(\sum_{i=1}^{k} F_{n+1,i},  \sum_{i=1}^{k} M_{n+1,i}, \eta_{n+1}\right)  - k g\left({\bf E}_{\boldsymbol\eta} F_{n+1,i},  {\bf E}_{\boldsymbol\eta} M_{n+1,i}, \eta_{n+1}\right) \right|^h\right)^{1/h}. 
\end{eqnarray*}
Due to~(\ref{ApproxLipBBP}) the first term is bounded from above by
\begin{eqnarray*}
c_1(\eta_{n+1}) \left({\bf E}_{\boldsymbol\eta}\left(\sum_{i=1}^{k} (F_{n+1,i}+ M_{n+1,i})\right)^{(1-\delta) h}\right)^{1/h}\le \\
c_1(\eta_{n+1})  k^{1-\delta} \left({\bf E}_{\boldsymbol\eta} \left(F_{n+1,1}+M_{n+1,1}\right)^{(1-\delta) h}\right)^{1/h}.
\end{eqnarray*}
Since $g$ is Lipschitz, the second term is bounded from above by
\begin{eqnarray*}
c_2(\eta_{n+1})^{1/h}\left({\bf E}_{\boldsymbol\eta}\left|\sum_{i=1}^{k} \left(F_{n+1,i}+M_{n+1,i} -{\bf E}_{\boldsymbol\eta} (F_{n+1,i}+M_{n+1,i})\right)\right|^h \right)^{1/h} \le \\
c_2(\eta_{n+1})^{1/h} c(h) k^{\max(1/h, 1/2)} \left({\bf E}_{\boldsymbol\eta}\left|F_{n+1,1}+M_{n+1,1}{\bf E}_{\boldsymbol\eta} (F_{n+1,1}+M_{n+1,1})\right|^h\right)^{1/h}.
\end{eqnarray*}
Thus,~(\ref{BnCond}) holds with $\widetilde{h}=\max(h/2, (1-\delta)h)$ and
\begin{eqnarray*}
\kappa_{n+1} = e^{-\xi_{n+1} h}
\left(c_1(\eta_{n+1})\left({\bf E}_{\boldsymbol\eta} \left(F_{n+1,1}+M_{n+1,1}\right)^{(1-\delta) h}\right)^{1/h} + \right.\\
\left.
c_2(\eta_{n+1})^{1/h} c(h) \left({\bf E}_{\boldsymbol\eta}\left|F_{n+1,1}+M_{n+1,1}-{\bf E}_{\boldsymbol\eta} (F_{n+1,1}+M_{n+1,1})\right|^h\right)^{1/h}
\right)^h.
\end{eqnarray*}
The condition ${\bf E}^{(h)} \kappa < +\infty$ is equivalent to 
\begin{eqnarray}
\label{BBPRECond}
{\bf E}\left(c_1^h (\eta_{1}){\bf E}_{\boldsymbol\eta} \left(F_{1,1}+M_{1,1}\right)^{(1-\delta) h}\right)<+\infty,\\
{\bf E}\left(c_2(\eta_{1}) {\bf E}_{\boldsymbol\eta}\left|F_{1,1}+M_{1,1}-{\bf E}_{\boldsymbol\eta} (F_{1,1}+M_{1,1})\right|^h\right) < +\infty.
\nonumber
\end{eqnarray}
Let
$$
\rho:=\lim_{n\to\infty} \sqrt[n]{{\bf P}(N_n=j|N_0=i)},
$$
and let $h^*_{+}: R(h^*_+)=\rho$ and $R'(h^*_+)>0$. By Theorem~\ref{YnSub} if for some $h>h^*_{+}$~(\ref{BBPRECond}) holds, then 
$$
{\bf P}(N_n>0)\sim c \rho^n,\ {\bf P}(N_n=k|N_n>0)\to p_k^*,\ k\in \mathbb{N},\ n\to\infty,
$$
where $\{p_k^*\}$ is some distribution. Note, that in the case of BBP condition~(\ref{BBPRECond}) holds if 
\begin{equation}
\label{BBPCond}
{\bf E}F_{1,1}^h<+\infty,\ {\bf E}M_{1,1}^h<+\infty.
\end{equation}

It looks like it is to hard to express $\rho$ explicitly, since the transition matrix $P^+$ has a complicated structure. However, we can give some upper and lower bounds. 

First of all, let $N_n$ be BBP. Then the subcritical process is strongly subcritical. Then
\begin{thm}
Let $\{N_{n}\}$ be a subcritical BBP, satisfying the approximation condition~(\ref{ApproxLipBBP}) and~(\ref{BBPCond}) holds for $h=h^*_{+}$. Then Theorem~\ref{YnSub} holds with some $\rho=\rho^*(1)\in (0,\mu]$, where $\mu = e^{\xi_1}$.
\end{thm}
\begin{proof}
Since the process is subcritical and $\xi$ is degenerate, $\rho^+=0$ and the process is strongly subcritical. The assumptions of Theorem~\ref{YnSub} were checked above. Let prove that $\rho\le \mu$.

By Lemma 5 of~\cite{Zhiyanov} we have
$$
{\bf E}N_n\le R(1)^n = \mu^{n} 
$$
and $\rho^*(1)\le R(1)$. Since $\rho^+=0$ we have $h^*_{++}=+\infty$. By Lemma~\ref{RhoRepresentation} $\rho^*(h)=\rho$ for any $h$. Thus, $\rho=\rho^*(1)\le R(1)=\mu$.
\end{proof}
In the case of BBPRE the situation is more complicated.
\begin{thm}
Let $\{N_n\}$ be a subcritical BBPRE, satisfying the approximation condition~(\ref{ApproxLipBBP}). Assume that $\widehat{h}>1$ and $R'(1)<0$. Then the process is strongly subcritical and $\rho = \rho^*(1)$. 

If~(\ref{BBPRECond}) holds for $h>h_{+}^*$, then Theorem~\ref{YnSub} holds.
\end{thm}
\begin{proof}
By Lemma 5 of~\cite{Zhiyanov} we have $\rho^*(1)\le R(1)$. If $\rho^*(1)<R(1)$, then $\rho^*(1)=\rho$ by Lemma~\ref{RhoRepresentation}. Otherwise, $\rho^*(1)=R(1)$ and $\rho^*(h)>R(h)$ for all $h\in (1,1+\delta)$ for some $\delta>0$ , since $\rho^*(h)-R(h)$ is the increasing function in some neighbourhood of $1$. Thus, $\rho^*(h)=\rho$ for all $h\in (1,1+\delta)$ and $\rho=\rho^*(1)=R(1)$ by continuity. 
\end{proof}
\begin{remark}
Thus, $R'(1)<0$ and~(\ref{BBPRECond}) for $h=\beta$ is a sufficient condition for Theorem~\ref{YnSub}. 
\end{remark}

\subsection*{Acknowledgment} 
We wish to thank G.A. Bakay for useful discussions, and professors V.A. Vatutin and V.I. Afanasyev for important remarks.

This work was supported by the Russian Science Foundation under grant no. 24-11-00037,
https://rscf.ru/en/project/24-11-00037/ and performed at Steklov Mathematical Institute of Russian Academy of Sciences.


\end{document}